\documentclass[12pt]{amsart}
\usepackage{amssymb,amsmath,amsthm,fullpage}
\usepackage{pdfsync,color,soul}
\usepackage{enumerate}
\usepackage{pdfsync}
\usepackage{mathrsfs}
\usepackage[bookmarksnumbered,colorlinks,plainpages,backref]{hyperref}
\usepackage[normalem]{ulem}

\newcommand{\R}{\mathbb R}

\newcommand{\N}{\mathbb N}

\DeclareMathOperator{\supp}{supp}

\newcommand{\p}{\partial}

\newcommand{\vp}{\varphi}

\newcommand{\atopp}[2]{\genfrac{}{}{0pt}{2}{#1}{#2}}
\newcommand{\nn}{\nonumber}

\newcommand{\la}{\langle}
\newcommand{\ra}{\rangle}

\DeclareMathOperator{\opG}{\mathcal G}
\DeclareMathOperator{\opE}{\mathcal E}

\newcommand{\Om}{\Omega}

\newtheorem{thm}{Theorem}[section]
\newtheorem{prop}[thm]{Proposition}
\newtheorem{lem}[thm]{Lemma}

\theoremstyle{definition}
\newtheorem{defn}[thm]{Definition}
\newtheorem{example}[thm]{Example}

\theoremstyle{remark}
\newtheorem{rem}[thm]{Remark}

\newcommand{\Gqm}{\mathcal{E}^{q,M}}

\DeclareMathOperator{\opF}{\mathcal F}

\newcommand{\Gqs}{\mathcal{G}^{q,s}}

\numberwithin{equation}{section}

\newcommand{\opS}{\mathcal{S}}

\newcommand{\opWF}{\mathcal{WF}}

\newcommand{\opX}{\mathfrak{X}}

\newcommand{\mfs}{\mathfrak{s}}
\newcommand{\mfS}{\mathfrak{S}}
\newcommand{\mfP}{\mathfrak{P}}
\newcommand{\mfp}{\mathfrak{p}}
\newcommand{\mfG}{\mathfrak{G}}

\begin{document}

\title{Distributions with Decay and Restriction Problems}

\author{Gustavo Hoepfner and Andrew Raich}

\thanks{This work was partially supported by FAPESP grant numbers 2017/03825-1 and 2018/02663-0.}

\address{Departamento de Matem\'atica, Universidade Federal de S{\~a}o Carlos, S{\~a}o Carlos, SP 13565-905, Brazil}
\email{hoepfner@dm.ufscar.br}

\address{Department of Mathematical Sciences, SCEN 309, 1 University of Arkansas, Fayetteville, AR 72701}
\email{araich@uark.edu}



\begin{abstract}
In this paper we introduce a new type of restriction problem, called the \textit{restriction problem with moments}.
We show that the  surface area measure of the sphere satisfies the $L^p$-$L^2$ restriction problem with
moments if $1 \leq p < \frac{2(d+2)}{d+3}$ and that the Frostman measure constructed by
Salem satisfies the $L^p$-$L^2$ restriction problem with moments if $1 \leq p < \frac{2(2-2\alpha+\beta)}{4(1-\alpha)+\beta}$ for certain values of
$\alpha$ and $\beta$.

The main tool to obtain these new type of restriction phenomenon is the notion of distributions with decay in connection with classes of global $L^q$ ultradifferentiable functions. 
We develop the notion of distributions with decay and 
use it to define global wavefront sets of classes of function spaces, including
$L^p$-Sobolev spaces on $\mathbb{R}^d$ as well as global $L^q$-Denjoy Carleman functions. 
We also introduce the corresponding notion of microglobal regularity. We prove 
a characterization of distributions (in a given function space) with decay in terms of microglobal regularity in every direction of their Fourier transforms.
\end{abstract}

\maketitle

%
%
\section{Introduction}\label{sec:intro}

Our motivating problem is the following: 
given a function space $\opX \subset \opS'(\R^d)$,  determine necessary and sufficient conditions on $\hat T \in \opS'(\R^d)$
so that $T\in \opX$, that is, to understand when
\begin{equation}\label{eqn:T = hat psi}
T = \hat\psi \quad \text{for some}\quad \psi\in \opX.
\end{equation}
This question leads us naturally to define
and explore a microglobal wavefront set defined in terms of the Fourier transform. Additionally, we apply this new technology
to the classical restriction problem and prove a new type of restriction phenomenon that we call the restriction problem with moments.
Essentially, the restriction problem with moments is the restriction problem applied to a measure (as in the classical case) and moments of the measure.
Moments of the measure correspond to derivatives of the Fourier transform, so we are, in effect, showing that the restriction problem can hold for 
certain tempered distributions. This is a new phenomenon.

The origin of our investigation is the paper of Boggess and Raich \cite{BoRa13h} in which they find conditions on the Fourier transform
of a function $f$ that guarantees that $f$ exhibits a specific type of exponential decay. This led to the development of function classes 
to explore this phenomenon more closely which in turn led to defining global $L^q$ Gevrey and global $L^q$ Denjoy-Carleman functions
\cite{AdHoRa17,HoRa19}. Our hope was to use these function classes to explore global properties of partial differential operators,
but we discovered that the Fourier transforms of such functions can be highly nonsmooth (e.g.,
a measure supported on a Salem set). To overcome
this difficulty, we substituted the FBI transform for the FBI transform and developed microglobal tools to analyze operators \cite{HoRa19g}.
While this program was successful for certain classes of problem, many objects are naturally analyzed with Fourier transforms, so in this paper, we
develop a microglobal analysis based on a theory of distributions with decay.

The Fourier transform interchanges smoothness with decay so function spaces defined in terms of smoothness estimates,
e.g., Sobolev spaces, global $L^q$ Denjoy-Carleman spaces, etc., ought to have characterizations in terms of the 
decay of the Fourier transform.
One complicating factor is that integrability conditions at $\infty$ are decay conditions, hence the Fourier transform
of an $L^1$ function is smoother than that of an $L^q$ function, $q>1$. Additionally,  Fourier transforms
of $L^q$ functions, regardless of their smoothness, need not be functions at all \cite{HoRa19}. In fact, our viewpoint
of restriction theorems is that Fourier transforms of $L^q$ functions, $1 \leq q < 2$ have sufficient additional
smoothness to allow them restrict to the support of a measure supported on an appropriate lower dimensional subset.

We work with two objects.
\begin{enumerate}
\item Distributions with decay. These distributions  capture the Fourier transforms of functions in
$\opX$.
\item The (Fourier) microglobal wavefront set. In \cite{HoRa19g}, we defined a notion of the wavefront set based on 
a global FBI transform. The FBI transform is difficult to compute and for objects naturally defined as tempered
distributions, the Fourier transform is more natural. Consequently, we develop a notion of microglobal regularity
and a Fourier transform based microglobal wavefront set.
\end{enumerate}

In the second part of the paper (Section \ref{sec:restriction}), 
we apply the distribution with decay machinery to measures supported on lower dimensional
subsets of $\R^d$, specifically, the Salem measure and the surface area measure of the sphere. By proving that the Fourier
transforms of these measures are certain global $L^q$-Gevrey functions, we can show that these measures, as well as all of their
moments satisfy the $L^p$-$L^2$ restriction problem for a range of $p$ that we compute.
We call this the$L^p$-$L^2$ restriction problem with moments. It is a generalization of and a stronger result than
the classical $L^p$-$L^2$ restriction problem.
Additionally, by keeping track of
the decay of the derivatives of the transform and the fact that the transform originate from a compactly supported measure, we are able to establish
the range  $1 \leq p < \frac{2(d+2)}{d+3}$ for the sphere and $1 \leq p <  \frac{2(2-2\alpha+\beta)}{4(1-\alpha)+\beta}$
for Salem's measure.
This is a new type of restriction theorem and shows that certain tempered distributions may also satisfy the restriction problem.

\subsection{Acknowledgements}
The authors would like to thank Malabika Pramanik, Andreas Seeger, and Jim Wright for looking at an early draft the paper and encouraging us to improve the results 
on the restriction problems.

%
%
\section{General Function Classes}
In order to motivate Definition \ref{defn:dist with decay}, we provide several examples.
\begin{example}[$L^q$ functions with polynomial decay of order $k$]
\label{ex:L^q with decay} Let $f\in L^q(\R^d)$, $1 < q < \infty$,  so that
\begin{equation}\label{eq:POLdecay}
\|x^\alpha f\|_{L^q(\R^d)} \leq C, \qquad \text{for all } |\alpha| \leq k.
\end{equation}
We  call such a function an \emph{$L^q$ function with polynomial decay of order $k$}. By H\"older's inequality and duality,
\eqref{eq:POLdecay} is equivalent to
\[
|\la x^\alpha f , \vp \ra| \leq C \|\vp\|_{L^p(\R^d)}
\]
for all $\vp \in \opS(\R^d)$.
\end{example}

\begin{example}[Fourier transforms of functions in a Sobolev space] \label{ex:W^k,q}
Let $1\leq q \leq \infty$, $k\geq 1$, and $f \in W^{k,q}(\R^d)$. Then if $|\alpha|\leq k$ and $\vp\in\opS(\R^d)$
\[
|\la \xi^\alpha \hat f, \vp \ra| = |\la D^\alpha f, \hat\vp\ra| \leq \|f\|_{W^{k,q}(\R^d)} \|\hat\vp\|_{L^p(\R^d)}
\]
where $\frac 1q + \frac 1p =1$. In this case, $\hat f$ is a distribution with decay, but it is fundamentally different than an $L^q$ function with decay because 
there is no reason to believe that $\hat f \in L^p(\R^d)$ (or even a function), unless, of course, $1 \leq q \leq 2$. Additionally, 
the decay is only apparent when
measured against $\hat\vp$. In this case, we say that
$\hat f$ is a \emph{distribution with polynomial decay of order $k$ with respect to the Fourier transform}.
\end{example}

\begin{example}[Fourier transforms of global $L^q$-Gevrey functions]
 \label{ex: L^q global Gevrey}
Let $1 \leq q \leq \infty$. We say that $f$ is a global $L^q$-Gevrey function and write  and $f \in \Gqs(\R^d)$ 
(see \eqref{eq:EqM} below for a more general framework for these function spaces) if
there exist constants $C,A>0$ so that $\|D^\alpha f\|_{L^q(\R^d)}\leq CA^{|\alpha|}|\alpha|^{|\alpha|s}$.
If $\vp \in \opS(\R^d)$, then
\[
|\la \xi^\alpha \hat f, \vp \ra| \leq CA^{|\alpha|}|\alpha|^{|\alpha|s} \|\hat \vp\|_{L^p(\R^d)}.
\]
This shows that $\hat f$ is also a distribution with decay with respect to the Fourier transform, and we now introduce a notation to quantify it.
\end{example}

\subsection{Semi-norms and function spaces}
In order to capture the widest array of function spaces and  tempered distributions, we define our spaces 
using the general classes of semi-norms that define the Schwartz class.

One of the most important features of the Fourier transform is its exchange of smoothness for decay, and vice versa, and this feature underlies all Paley-Wiener Theorems. Exactly how to capture
 this interplay is challenging, and we focus on using classes of semi-norms. 
One of the flexible features of $\opS(\R^d)$ is that it can be defined using $L^p$ norms for
any $1 \leq p \leq \infty$, either on the function side or the Fourier transform side.
In \cite{BoRa13h,AdHoRa17}, this type of duality was used to capture exponential decay in terms of inequalities on the Fourier transform.

We notate a \emph{parameter set} $\mfS$ by 
\[
\mfS = \big\{(\alpha,\beta,q) : \alpha,\beta\in\N_0^d, 1 \leq q \leq \infty\big\}
\]
and for simplicity, we assume that $q$ is fixed over the parameter set. We impose the additional requirement that if
$(\alpha,\beta,q)\in\mfS$, then so is $(\alpha',\beta',q)$ for all $\alpha'\subset\alpha$ and $\beta'\subset\beta$.
Given a parameter set $\mfS$, we also have an associated \emph{growth set} 
\[
\mfG = \{a_{\mfp}: \mfp\in\mfS\} \subset (0,\infty)
\]
and call the pair $(\mfS,\mfG)$ a \emph{parameter pair}. 
For an element $(\mfs,a_\mfs) \in (\mfS,\mfG)$ and $f\in\opS(\R^d)$, define the semi-norm
\[
\rho(f;\mfs;a_\mfs) = \rho_{\mfs,a_{\mfs}}(f) 
=  a_{\mfs} \|x^\beta D^\alpha f\|_{L^q(\R^d)}.
\]

For a more general class, we could have our norms defined on $\Om\subset\R^d$ and also include compact sets $K\subset\Om$ with
$\mfS = \{(\alpha,q,\Om,K) : \alpha,\beta \in\N_0^d, 1 \leq q \leq \infty, \Om\subset\R^d\}$
and $\rho(f;\mfs;a_\mfs) = a_{\mfs} \|x^\beta D^\alpha f\|_{L^q(K)}$. This would have allowed us to
recover the classical Denjoy-Carleman spaces, i.e., with $q=\infty$, but we do not pursue these
spaces in this work because 
it is known that these spaces possess a characterization via the Fourier transform, see \cite{HoMe18}.

Our interest is in function spaces $\opX\subset\opS'(\R^d)$ that are defined by
a finite or countable collection of semi-norms
$\rho_{\mfs,a_{\mfs}}$ or an increasing/decreasing union of such spaces. 

For concreteness,  we are going to concentrate on
the cases when
$\opX$ is defined by a finite collection of semi-norms (e.g., Examples \ref{ex:L^q with decay} and \ref{ex:W^k,q}), a global  $L^q$ Gevrey space \cite{AdHoRa17} (e.g., Example \ref{ex: L^q global Gevrey}),
or a global $L^q$ Denjoy-Carleman space as in  \cite{HoRa19g} (see (\ref{eq:EqM}) below). 

If $\opX$ is defined by a finite index set $\mfS$,
\begin{equation}\label{eqn:X finite}
\opX = \{ f \in \opS'(\R^d) : \rho(f;\alpha,\beta,q;C) < \infty,\ \text{for all }(\alpha,\beta,q)\in\mfS \text{ and some }C>0\},
\end{equation}
and the growth set $\mfG$ plays no role.

For function spaces $\opX$ defined by an infinite collection of semi-norms $\mfS$,
we  concentrate on function classes for which there exist increasing sequences of positive numbers 
$M=(M_j)_{j\in\N_0}$ and $M'=(M_j')_{j'\in\N_0}$
so that
\begin{equation}\label{eq:faster}
j! \le \min\{M_j,M_j'\}\qquad \forall\, j\in\N_0,
\end{equation}
\begin{equation}
\opX_A = \Bigg\{ f \in \opS'(\R^d) : 
\bigg\{ \frac{\|x^\beta D^\alpha f\|_{L^q(\R^d)}}{A^{|\alpha|+|\beta|}M_{|\alpha|}M'_{|\beta|}}:
(\alpha,\beta,q)\in\mfS\bigg\}
\in \ell^q \Bigg\} \label{eqn:opX_A} 
\end{equation}
and
\begin{align}
\opX &= \bigcup_{A>0} \opX_A \nn\\
&= \Big\{ f \in \opS'(\R^d) : \text{ there exist }C,A>0 \text{ so that } \nn\\
&\hskip3cm 
\|x^\beta D^\alpha f\|_{L^q(\R^d)} \leq C A^{|\alpha|+|\beta|}
M_{|\alpha|}M'_{|\beta|} \|f\|_{L^q(\R^d)},\ (\alpha,\beta,q,j)\in\mfS\Big\}.
\label{eqn:opX infinite}  
\end{align}

\begin{rem}\label{rem:differentiation,mult switching}
Since the number of terms with $\alpha'\subset\alpha$ 
and $\beta'\subset\beta$ grows geometrically in $|\alpha|$ and $|\beta|$ and $(M_k)_{k\in\N_0}$
is an increasing sequence, we can estimate
$D^\beta\{x^\alpha f\}$
without changing $M_k$ but merely increasing the geometric constant $A$.
\end{rem}

\subsection{Distributions with decay}
%
\begin{defn}\label{defn:dist with decay}
Let $T \in \opS'(\R^d)$ and $(\mfP,\mfG)$ be a parameter pair. We say that $T$ is a 
\emph{tempered distribution with $(\mfP,\mfG)$ decay} if for all $\vp\in\opS(\R^d)$ it holds that
\begin{equation}\label{eqn:dist decay}
\Big\{\frac{|\la x^\beta D^\alpha T,   \vp\ra|}{c_\mfp \|\vp\|_{L^p(\R^d)}} : (\alpha,\beta,q) = \mfp \in \mfP,\ c_\mfp \in \mfG \Big\} \in \ell^q
\end{equation}
and that $T$ is a \emph{tempered distribution with $(\mfP,\mfG)$ decay in the direction $x\neq 0$} if there exists
an open cone $\Gamma$ containing $x$  so that \eqref{eqn:dist decay} holds for all $\vp\in\opS(\R^d)$ with
$\supp\vp \subset \Gamma$.

We say that $T$ is a \emph{tempered distribution with $(\mfP,\mfG)$ decay at $x=0$}  if \eqref{eqn:dist decay} holds with $\Gamma$ replaced by $B(0,1)$.

Also, we say that $T$ is a
\emph{tempered distribution with $(\mfP,\mfG)$ decay with respect to the Fourier transform} if
\begin{equation}\label{eqn:dist decay, Ft}
\Big\{\frac{|\la x^\beta D^\alpha T,   \vp\ra|}{c_\mfp \|\hat\vp\|_{L^q(\R^d)}} : (\alpha,\beta,q) =\mfp \in \mfP,\ c_\mfp \in \mfG \Big\} \in \ell^q
\end{equation}
and that $T$ is a \emph{tempered distribution with $(\mfP,\mfG)$ decay in the direction $x\neq0$ (resp., $x=0$)) with respect to the
Fourier transform} if there exists
an open cone $\Gamma$ containing $x$ (resp., $B(0,1)$) so that \eqref{eqn:dist decay, Ft} holds for all $\vp\in\opS(\R^d)$ with
$\supp\vp \subset \Gamma$ (resp., $\supp\vp\subset B(0,1)$).
\end{defn}

To motivate the correct notion of microglobal regularity, observe that if $f \in \Gqs(\R^d)$, then
\[
|\la \xi^\alpha \hat f, \check \vp \ra| = |\la D^\alpha f, \vp \ra | \leq C A^{|\alpha|} |\alpha|^{|\alpha|s}\|\vp\|_{L^p(\R^d)}
\]

Given a parameter set $\mfP$, we define the \emph{dual parameter set} $\hat\mfP$ by
$\hat\mfP = \{(\alpha,\beta,q) \in \N_0\times\N_0\times[1,\infty] : (\beta,\alpha,p)\in \mfP\}$ where $\frac 1p + \frac 1q=1$.

\begin{defn}\label{defn:Fourier wavefront, X}
Let $T \in \opS'(\R^d)$, $\xi\in\R^d$, and $(\mfP,\mfG)$ be a parameter pair. We say that $T$ is 
\emph{Fourier microglobally regular in the direction $\xi\in\R^d$ with respect to the parameter pair $(\mfP,\mfG)$} 
(or simply Fourier $(\mfP,\mfG)$-microglobal regular in the direction $\xi\in\R^d$), 
if $\hat T$ is a tempered distribution with $(\hat\mfP,\mfG)$ decay in the direction $\xi$ (possibly with respect
to the Fourier transform).

We define the \emph{Fourier wavefront set of $T$ with respect to the parameter pair $(\mfP,\mfG)$}
or the Fourier $(\mfP,\mfG)$-wavefront set, denoted by  $\opWF_{\mfP,\mfG}(T)$, by
\[
\opWF_{\mfP,\mfG}(T) = \{\xi\in\R^d : T \text{ is not Fourier $(\mfP,\mfG)$-microglobal regular in the direction } \xi\}.
\]
If $\opX$ is a function space defined by the parameter pair $(\mfP,\mfG)$, that is, 
\[
\opX = \{ f\in L^q(\R^d) : \rho(f;\mfs;c_{\mfs}) < \infty \text{ for all }\mfs\in\mfP\}
\]
then we will call the $(\mfP,\mfG)$-wavefront set by \emph{$\opX$-wavefront set} and denote it by $\opWF_{\opX}$. Similarly,
we say that a distribution is $\opX$ microglobal regular in the direction $\xi$. Finally, for a given space $\opX$, whether or not the notion
of wavefront set, decay, etc. is taken with respect to the Fourier transform  is determined from the definition of $\opX$.
\end{defn}

\begin{thm}\label{thm:Gevrey and FT, general X}
Let $T\in \opS'(\R^d)$, $1 < q < \infty$, and $\opX\subset\opS'(\R^d)$ be defined by the parameter pair $(\mfS,\mfG)$. 
There exists $\psi \in \opX$
so that $\hat T = \psi$ if and only if the $\opWF_{\opX}(T) = \emptyset$.
\end{thm}

\begin{proof}
The forward direction follows from Remark \ref{rem:differentiation,mult switching} and the inequality
\[
\big|\big(x^\alpha D^\beta T, \vp\big)\big|
= \big|\big(x^\alpha D^\beta \hat\psi, \vp\big)\big|
= \big|\big(D^\alpha \{x^\beta \psi\}, \hat\vp\big)\big|
\leq \sum_{\beta'\subset\alpha\cap\beta} \binom{\beta}{\beta'} \frac{\beta!}{(\beta-\beta')!}\big(x^{\beta-\beta'} D^{\alpha-\beta'}\psi,\hat\vp\big).
\]
Now suppose that $T \in \opS'(\R^d)$ is a tempered distribution with 
$(\hat\mfS,\mfG)$-decay given by \eqref{eqn:dist decay, Ft}.
We know that $\xi^\alpha$ is a smooth, slowly increasing function, and we use the argument of the forward direction to establish
\[
\big|\big\la x^\beta D^\alpha\hat T,\vp\big\ra\big|
= \big|\big\la D^\beta\{\xi^\alpha T\}, \hat \vp \big\ra\big| 
\leq c_{\alpha,\beta,q} \|\vp\|_{L^p(\R^d)}
\]
since $\hat{\hat\vp}(x) = \vp(-x)$.
Since $1 < p < \infty$, we know that $\opS(\R^d)$ is dense in $L^p(\R^d)$ and that $L^p(\R^d)' = L^q(\R^d)$. Consequently,
\[
\vp \mapsto \big\la x^\beta D^\alpha\hat T,\vp\big\ra
\]
extends to a bounded linear operator on $L^p(\R^d)$. Therefore, there exists $\psi_{\alpha,\beta} \in L^q(\R^d)$ so that
\[
\big\la x^\beta D^\alpha\hat T,\vp\big\ra = \int_{\R^d} \psi_{\alpha,\beta}(\xi) \vp(\xi)\, dV(\xi).
\]
We claim that
\[
\psi_{\alpha,\beta}(\xi) = x^\beta D^\alpha\psi_{0,0}(\xi).
\]
If the claim holds, then
\[
\|x^\beta D^\alpha \psi_{0,0}\|_{L^q(\R^d)} \leq C_{|\alpha|,|\beta|},
\]
for all $(\alpha,\beta,q)\in\mfS$. This would mean $\psi_{0,0}\in \opX$. To prove the claim we first assume $\beta=0$ and 
observe that $\psi_{0,0}\in L^q(\R^d)$ and hence
for any $\vp\in \opS(\R^d)$,
\[
\la D^\alpha\psi_{0,0}, \vp \ra = (-1)^{|\alpha|} \la \psi_{0,0}, D^\alpha\vp\ra = (-1)^{|\alpha|} \la \hat T, D^\alpha\vp \ra
= \la D^\alpha \hat T, \vp \ra = \la \psi_{\alpha,0},\vp\ra.
\]
Again by density, it must be the case that $D^\alpha\psi_{0,0} = \psi_{\alpha,0}$. The proof that 
$\psi_{\alpha,\beta} = x^\beta \psi_{\alpha,0}$ is similar.
\end{proof}
\begin{rem}
\begin{enumerate}[(1)]
\item The forward direction of this proof  holds for $1 \leq q \leq \infty$. It is only the reverse direction
that eliminates $q=1,\infty$ from the theorem.
\item
If $1 < p < 2$, $\frac 1q+\frac 1p = 1$, and $\opX = \Gqs(\R^d)$, 
then the Hausdorff-Young inequality implies that
\[
|\la  \xi^\alpha T, \vp \ra| \leq C A^{|\alpha|}|\alpha|^{|\alpha|s}\|\hat\vp\|_{L^p(\R^d)}  \leq C A^{|\alpha|}|\alpha|^{|\alpha|s} \|\vp\|_{L^q(\R^d)}
\]
which, of course, means that $T$ agrees with a function that decays exponentially in the $L^p$ sense (appropriately defined, see also \cite{BoRa13h,AdHoRa17}).
\end{enumerate}
\end{rem}

\begin{thm}\label{thm:Gevrey and microglobally regular, general X}
Let $T\in \opS'(\R^d)$ and $\opX\subset\opS'(\R^d)$ be a function space defined by the parameter pair $(\mfS,\mfG)$.  
Then $T$ is given by integration against a function  $\psi\in\opX$ if and only if
$T$ is Fourier $\opX$-microglobally regular with respect to $(\mfS,\mfG)$ for all directions $\xi\in\R^d$.
\end{thm}

\begin{proof} The forward direction follows immediately from Theorem
\ref{thm:Gevrey and FT, general X}. Therefore, we may assume that $T$ is 
Fourier $\opX$-microglobally regular  in every direction $\xi\in\R^d$.
Let $\chi_1,\dots,\chi_N$ be a smooth partition of unity on $\mathbb{S}^{d-1}$ so that $\supp\chi_j \subset\Gamma_j$ where
$\Gamma_j$ is a cone given by Definition \ref{defn:Fourier wavefront, X} corresponding to some $\xi_j \in\R^d\setminus\{0\}$.
We extend $\chi_j$ to  $\R^d\setminus\{0\}$ homogeneously of degree $0$. We also let $\chi \in C^\infty_c(\R^d)$
so that $\chi \equiv 1$ on $B(0,1/2)$ and $\supp\chi \subset B(0,1)$. Let $\vp\in\opS(\R^d)$ and write
\begin{equation}\label{eqn:decomposition, general X}
\check\vp = \chi\check \vp + (1-\chi)\sum_{k=1}^N \chi_k\check \vp.
\end{equation}
Let $(\alpha,\beta,q)=\mfs \in\mfS$ so that $(\beta,\alpha,p)\in \widehat\mfS$.
Since $(1-\chi)\chi_k\check\vp \in\opS(\R^d)$ satisfies $\supp (1-\chi)\chi_k\check\vp  \subset \Gamma_k$, we estimate
\[
|\la x^\beta D^\alpha T, \opF\{(1-\chi)\chi_k\check \vp\} \ra|
= |\la D^\beta \{\xi^\alpha \hat T\}, (1-\chi)\chi_k\check \vp \ra|
\leq c_{\mfs} 
\|\opF\{(1-\chi)\chi_k\check\vp\}\|_{L^p(\R^d)} 
\]
for the appropriate $C_{|\alpha|,|\beta|}$, depending on the definition of $\opX$.
Moreover, $(1-\chi)\chi_k\in L^\infty$ is a Mikhlin multiplier, so its Fourier transform is the convolution kernel of an $L^p$-bounded operator, $1 < p < \infty$.
Hence
\[
\|\opF\{(1-\chi)\chi_k\check\vp\}\|_{L^p(\R^d)} 
= \| \widehat{(1-\chi)\chi_k}*\vp\|_{L^p(\R^d)}
\leq C\|\vp\|_{L^p(\R^d)}.
\]
The argument for decay at $0$ follows from the same argument with $\chi$ replacing by $(1-\chi)\chi_k$.
The theorem now follows from the decomposition of $\vp$ given in \eqref{eqn:decomposition, general X}.
\end{proof}

%
%
\section{A special case: the space of global $L^q$ Denjoy-Carleman functions}
\label{sec:Denjoy-Carleman}

If $M=(M_j)_{j\in\N_0}$ is an appropriately defined increasing sequence, then the global $L^q$ Denjoy-Carleman space
\[
\Gqm(\Om) = \bigcup_{A>0} \Gqm_A(\Om) 
\]
where
\begin{align}
\Gqm_A(\Om) &= \big\{f \in C^\infty(\R^d) : \text{there is a constant }C>0  \text{ so that}\label{eq:EqM} \\
&\hskip3.5cm \|D^\alpha f\|_{L^q(\Om)} \leq C A^{|\alpha|}M_{|\alpha|} 
\text{ for all multi-indices }\alpha\big\}.\notag
\end{align}

We will not need to know the requirements on $M_j$ apart from that it is increasing and satisfies \eqref{eq:faster} (see \cite{HoRa19g}). The most important example of the
global $L^q$ Denjoy-Carleman spaces are the global $L^q$-Gevrey spaces 
$\Gqs(\R^d)$ in which case $M_j = j^{js}$ \cite{AdHoRa17,HoRa19}.

\begin{prop}\label{prop:regular at xi, regular at 0} 
Let $1 \leq q \leq \infty$ and $\opX$ be a function space contained by $W^{1,q}(\R^d)$.
There exists $T\in \opS'(\R^d)$ so that $T$ is $\opX$-microglobally regular for all $\xi\in\R^d\setminus\{0\}$ but not
$\opX$-microglobally regular at $\{0\}$.
\end{prop}

\begin{rem}
Examples for $\opX$ include $W^{k,q}(\R^d)$, $1 \leq k \leq \infty$, and $\Gqm(\R^d)$.
\end{rem}

\begin{proof} Let $T\in\opS'(\R^d)$ be defined by
\[
\la T, \vp \ra = \int_{\R^d} x_1 \vp(x)\, dV(x)
\]
Then
\[
\hat T = i \frac{\p \delta_0}{\p \xi_1}.
\]
Since $\supp \hat T = \{0\}$, it follows immediately that $T$ is $\opX$-microglobally regular for all $\xi\in\R^d\setminus\{0\}$.
However, $T$ is clearly not $\opX$-microglobally regular at $\{0\}$ because 
\[
\la T,\vp \ra = \la T, \hat{\check{\vp}}\ra
= \la \hat T, \check\vp \ra = i \Big\la \frac{\p \delta_0}{\p \xi_1}, \check\vp \Big\ra
= -i \frac{\p\check\vp(0)}{\p\xi_1}
\]
is certainly  bounded by neither $\|\check \vp\|_{L^p(\R^d)}$ nor $\|\vp\|_{L^p(\R^d)}$ for any $1 \leq p \leq \infty$.
\end{proof}

For the global $L^q$ Denjoy-Carleman spaces, we
now have two notions of the microglobal wavefront set -- the Fourier $\Gqm$-microglobal wavefront set and the FBI 
$\Gqm$-microglobal wavefront set. The former is defined on 
$\opS'(\R^d)$ and the latter on $\Gqm(\R^d)'$. It is a natural question to know whether or not these notions agree for distributions in
$\opS'(\R^d)\cap\Gqm(\R^d)'$. First, though, 
we recall the FBI transform and FBI $\Gqm$-microglobal regularity \cite{HoRa19g}
but need to introduce one piece of notation. For  $\xi\in\R^d$, define
\[
\la\xi\ra = \sqrt{1+\xi\cdot\xi}
\]
and the form $\omega$ by
\[
\omega = dx_1\wedge \cdots \wedge dx_d \wedge d(\xi_1+ix_1\la\xi\ra) \wedge \cdots d(\xi_d+ix_d\la\xi\ra)
:= \alpha(x,\xi) dx_1\wedge \cdots \wedge dx_d \wedge d\xi_1\wedge \cdots \wedge d\xi_d.
\]
The function $\alpha$ is a sum of terms of the form
\[
\alpha_\beta(x,\xi) = c_\beta x^\beta \Big(\frac{\xi}{\la\xi\ra}\Big)^\beta
\]
where $\beta \in \{0,1\}^d$.

Finally,  if $M=(M_j)_{j\in\N}$ is an appropriately defined increasing sequence, then its associated function $M(t)$ is given by
\[
M(t)=\sup_{p\in\N}\log \frac{t^p}{M_p}.
\]

Recall that the $\opG^{q,\frac 12}(\R^d)$ is set of global $L^q$ Gevrey functions of order $\frac12$.
\begin{defn}\label{defn:FBI transform}
Let $\Gqm(\R^d)$ be a global $L^q$ Denjoy-Carleman function class.
We say that $\Gqm(\R^d)$ \emph{completely contains} $\opG^{q,\frac 12}(\R^d)$ 
if $\Gqm_A(\R^d) \supset \opG^{q,\frac 12}(\R^d)$ for every $A>0$.

Given a global $L^q$ Denjoy-Carleman class $\Gqm(\R^d)$ that completely contains $\opG^{q,\frac 12}(\R^d)$
and a distribution $u\in \Gqm(\R^d)'$, define the \emph{FBI transform} of $u$
by
\[
\mathfrak{F} u(x,\xi) = \big\la u, e^{i(x-\cdot)\cdot\xi - \la\xi\ra(x-\cdot)^2} \alpha(x-\cdot,\xi)\big\ra
\]
The function $\mathfrak{F} u$ is well defined for $u\in\Gqm_A(\R^d)'$ since the exponential function $e^{-a|x|^2}\in\opG^{q,\frac 12}(\Om)$
 \cite[Section 4.1]{AdHoRa17}.
\end{defn}
\begin{defn}\label{defn:wfs}
Let $u\in {\Gqm}(\R^d)'$ and $\xi^0 \in \R^d\setminus \{0\}.$ We say that $u$ is 
\emph{FBI $\Gqm$-microglobal regular at $\R^d\times\{\xi^0\}$} (or simply $\xi^0$) if 
there exist a conic neighborhood  $\Gamma_0$ of $\xi^0$ in  $\R^d \setminus \{ 0\}$ and constants $c, C>0$ 
such that for each $q\le r\le \infty$, 
\begin{equation} \label{eq:fbi}
\|D^J_x \mathfrak{F} u(x,\xi)\|_{L^r(\R^d)} \leq C A_0^{|J|}M_{|J|} e^{-\frac{1}{c}M(a |\xi|)}, \quad \forall \xi\in\Gamma_0.
\end{equation}
We define the \emph{FBI $\Gqm$-wave front set of $u$} 
as the complement of the set of the directions $\xi$ in which $u$ is FBI $\Gqm$-microglobal regular, that is
\begin{eqnarray*}
WF_{\Gqm}(u) := \{\xi \in\R^d: u\, \mbox{ is not  $\Gqm$-microglobal regular at } \, \xi \}.
\end{eqnarray*}
\end{defn}

\begin{prop}\label{prop:wave front sets do not agree}
There exists $T \in \opS'(\R^d) \cap \opE^{1,M}(\R^d)'$ so that
\[
\opWF_{\opE^{1,M}}(T) \neq WF_{\opE^{1,M}}(T).
\]
\end{prop}

\begin{proof} Let $T$ be the distribution that is given by integration against $1$. 
To compute $\opWF_{\opE^{1,M}}(T)$, we observe that if
$\hat\vp\in \opS(\R^d)$ is supported in an open cone then the fact that 
$\hat T = \delta_0$ is supported at $\{0\}$ forces
\[
|\la D^\alpha T,\vp \ra| = |\la \delta_0, \xi^\alpha\check\vp \ra| =0.
\]
Hence, $\xi \notin\opWF_{\opE^{1,M}}(T)$ for all $\xi\neq 0$. On the other hand, for $\vp\in\opS(\R^d)$ with
$\supp\hat\vp\subset B(0,1)$, we observe that
\[
\la T,\vp\ra = \int_{\R^d} \vp(x)\,dV(x) = \hat\vp(0)
\]
and
\[
\hat\vp \mapsto \hat\vp(0)
\]
does not extend to a bounded linear operator on $L^1(\R^d)$. Hence $0 \in \opWF_{\opE^{1,M}}(T)$.

We now compute $WF_{\opE^{1,M}}(T)$. Note, though, that
\[
f \mapsto \int_{\R^d}\, f\, dx
\]
is a bounded linear operator $\opE^{1,M}(\R^d)$ (but not on $\Gqm(\R^d)$ if $q>1$).
We observe that
\begin{align*}
\mathfrak{F}T(x,\xi)
&= \int_{\R^d} e^{i(x-y)\cdot\xi - \la\xi\ra(x-y)^2} \alpha(x-y,\xi)\, dV(y) \\
&= (-1)^d \int_{\R^d} e^{iy\cdot\xi - \la\xi\ra y^2} \alpha(y,\xi)\, dV(y) \\
&= (-1)^d \sum_{\beta\in\{0,1\}^d} \bigg\{\Big(\frac{\xi}{\la\xi\ra}\Big)^\beta c_\beta
\prod_{k=1}^d \int_{\R} e^{iy_k\xi_k-\la\xi\ra y_k^2} y_k^{\beta_k}\, dy_k\bigg\}.
\end{align*}
Since $\beta_k=0$ or $1$, we easily compute that if $b>0$, then
\[
\int_{\R} e^{iya-by^2}\, dy = \frac{\sqrt\pi}{\sqrt b} e^{-\frac{a^2}{4b}}
\qquad\text{and}\qquad
\int_{\R} e^{iya-by^2} y\, dy = -\frac{i a \sqrt\pi}{2b^{3/2}} e^{-\frac{a^2}{4b}}.
\]
The result of this calculation is that the FBI transform of $T$ is \emph{independent} of $x$, hence there is no
way the resulting function is in $L^q(\R^d)$ for $1 \leq q < \infty$. Hence $WF_{\opE^{1,M}}(T) = \R^d\setminus\{0\}$.
\end{proof}

%
%
\section{Applications to the Restriction Problem}\label{sec:restriction}
We follow \cite[Chapter VIII.4]{Ste93}.
Let $S\subset\R^d$ be a subset supporting a measure $d\sigma$ on $S$ (possibly the induced Lebesgue
measure). We say that $S$ has the \emph{$L^p$-$L^q$ restriction property} if there exists $A_{p,q}>0$
so that
\[
\bigg(\int_S|\hat f(\xi)|^q\, d\sigma(\xi)\bigg)^{1/q} \leq A_{p,q}\|f\|_{L^p(\R^d)}, \quad\text{for all } f\in\opS(\R^d).
\]

For $\xi \in S$, the restriction operator
\[
Rf(\xi) = \int_{\R^d} e^{-ix\cdot\xi}f(x)\, dx.
\]
If $R^*$ is the adjoint operator and
\[
Uf(x) := R^* Rf(x) = \int_{\R^d}\int_S 	e^{i\xi(x-y)}\, d\sigma(\xi)\, f(y) \, dy = f*\check{d\sigma}(x)
\]
then Stein points out that the following are equivalent \cite{Ste93}:
\begin{enumerate}[(i)]
\item $R:L^p(\R^d)\to L^2(S,d\mu)$ is bounded (that is, the $L^p$-$L^2$ restriction property holds);
\item $U:L^p(\R^d)\to L^{p'}(\R^d)$ is bounded where $\frac 1p + \frac 1{p'}=1$;
\end{enumerate}

\begin{defn}\label{defn:moments restriction problem}
Fix $1 \leq p < \infty$.
Let $\opX \subset \opS'(\R^d)$ be a function space defined by a parameter pair $(\mfP,\mfG)$.
If $T$ is a tempered distribution so that $f * \hat T \in \opX$ for all 
$f\in L^p(\R^d)$, then we say
that \emph{$L^p$-$L^2$ restriction problem with $\opX$-moments} holds for $T$.
\end{defn}
We use the terminology moment in Definition \ref{defn:moments restriction problem} because if $\opX$ is a function space
defined by smoothness estimates such as Sobolev spaces or global $L^q$-Denjoy Carleman space, then 
estimating $D^\alpha \widehat{d\sigma}$ gives us information about $x^\alpha\, d\sigma$, the moments of $d\sigma$.

\begin{rem} If $L^{p'}(\R^d) \subset\opX$, then $T$ satisfies the classical $L^p$-$L^2$ restriction problem.
\end{rem}

\subsection{The Salem Example} In \cite{HoRa19}, we showed that for each $q>2$, the measure $d\mu$ constructed
by Salem \cite{Sal51} satisfies $\|D^k \widehat{d\mu}\|_{L^q(\R)}\leq C$ for some constant $C$. We proved this by showing
$\|D^k \widehat{d\mu}\|_{L^\infty(\R)}\leq 1$ and $|D^k\widehat{d\mu}(\xi)| \leq C|\xi|^{-\beta/2}$ for some $\beta$ that does not depend on $k$. 
It is also well known that the measure is supported
on a set of Hausdorff dimension $\alpha$ (which can be made arbitrarily close to, but greater than, $2/q$).
The latter two facts means that we
can employ Mockenhaupt's adaption of the Tomas-Stein argument to prove the following theorem.
\begin{thm}\label{thm:Salem and restriction}
The $L^p$-$L^2$ restriction problem with $\opG^{p',0}_1(\R)$-moments holds for the Salem measure when
$1 \leq p < \frac{2(2-2\alpha+\beta)}{4(1-\alpha)+\beta}$.
\end{thm}

\begin{proof}The proof follows by the argument of the proof of Theorem \ref{thm:surface area measure is global} or by combining
\eqref{eqn:eliminating alpha} with the proof of \cite[Theorem 4.1]{Moc00}.
\end{proof}

\subsection{Surface area measure on the sphere}
For each $m\in\N$, denote the $m$th Bessel $J$ function by
\[
J_m(r)=\frac{1}{2\pi}\int_0^{2\pi} e^{ir\sin\theta} e^{-im\theta}\, d\theta.
\]
\begin{thm}\label{thm:surface area measure is global}

Let $d\sigma$ be the surface area measure on sphere $S^{d-1}$. Then
\begin{enumerate}[1.]
\item
\[
\widehat{d\sigma}(\xi) = 2 \pi^{\frac d2} \frac{J_{\frac{d-2}2}(|\xi|)}{|\xi|^{\frac{d-2}2}} \in \opG^{q,1}(\R^d)
\]
if and only if $q > \frac{2d}{d-1}$.
\item The sphere satisfies the $L^p$-$L^2$ restriction problem with $\opG^{p',1}(\R^d)$-moments when
$1 \leq p < \frac{2(d+1)}{d+3}$.
\end{enumerate}
\end{thm} 

\begin{rem} The Bessel $J$ functions decay for large $x$ \emph{and} oscillate. If we ignored the oscillation and applied
Young's inequality to the to the estimates in
Part 1 of Theorem \ref{thm:surface area measure is global}, then we could only say that the
$L^p$-$L^2$ restriction problem holds with $\opG^{q,1}(\R^d)$-moments for 
for $1 \leq p <  \frac{4d}{3d+1}$. 
\end{rem}

The next lemma is the main combinatorial lemma that we use to control the derivatives of $\widehat{d\sigma}$.
\begin{lem}\label{lem:derivs of f}
Let $f_m(r) =  \frac{J_m(r)}{r^m}$. Then 
\begin{align}\label{eqn:form of f'}
f_m^{(2k)}(r) &= \sum_{j=0}^k (-1)^{j+k} a_{jk}r^{2j} f_{m+k+j}(r) \\
f_m^{(2k+1)}(r) &= \sum_{j=0}^k (-1)^{j+k+1} b_{jk} r^{2j+1}f_{m+k+j+1}(r) \nn
\end{align}
for some positive constants $a_{jk},b_{jk}$ that grow no faster than $C A^k k!$. Moreover,
\[
a_{kk} = b_{kk} = 1
\]
for all $k$ and the constants satisfy the recursive relationships that for $k\geq 1$,
\[
b_{jk} = \begin{cases} 2(j+1)a_{j+1,k}+a_{jk} & \text{if } 0 \leq j \leq k-1 \\ 1 &\text{if }j=k \end{cases}
\]
and
\[
a_{jk} = \begin{cases} b_{0,k-1} & \text{if } j=0 \\ b_{j,k-1}(2j+1) + b_{j-1,k-1} & \text{if } 1 \leq j \leq k-1 \\ 1 &\text{if }j=k \end{cases}.
\]
\end{lem}

\begin{proof}
To prove the lemma, we recall the well known properties
\[
f_m'(r) = -r f_{m+1}(r) \qquad\text{and}\qquad
f_m''(r) = r^2f_{m+2}(r) - f_{m+1}(r), \quad \text{for all }  m\in\N_0.
\]
Moreover, assuming the formula holds for $f^{(2k)}(r)$ and suppressing the variable of $f$, 
\begin{align*}
f_m^{(2k+1)} &= (-1)^{k+1} a_{0k} r f_{m+k+1} + \sum_{j=1}^k\big[(-1)^{j+k}a_{jk} (2j)r^{2j-1} f_{m+k+j} + (-1)^{j+k+1}a_{jk} r^{2j+1}f_{m+k+j+1}\big] \\
&=   \sum_{j=0}^{k-1}(-1)^{j+k+1} \big[2(j+1)a_{j+1,k}+a_{jk}\big] r^{2j+1}f_{m+k+j+1} -a_{kk} r^{2k+1}f_{m+2k+1}
\end{align*}
which yields
\[
b_{jk} = 2(j+1)a_{j+1,k}+a_{jk},\ 0 \leq j \leq k-1,
\quad\text{and}\quad
b_{kk} = a_{kk}
\]
Similarly, assuming the formula holds for $f_m^{(2k-1)}$, we compute
\begin{align*}
f_m^{(2k)}(r) &= \sum_{j=0}^{k-1} (-1)^{j+k} b_{j,k-1} r^{2j+1}f_{m+k+j}(r) \\
&=  \sum_{j=0}^{k-1} \big[ (-1)^{j+k} b_{j,k-1}(2j+1) r^{2j}f_{m+k+j} + (-1)^{j+k+1} b_{j,k-1} r^{2(j+1)}f_{m+k+j+1}\big] \\
&= (-1)^{k} b_{0,k-1}f_{m+k}  + \sum_{j=1}^{k-1}  (-1)^{j+k}\big[ b_{j,k-1}(2j+1) + b_{j-1,k-1}\big] r^{2j}f_{m+k+j} + b_{k-1,k-1} r^{2k}f_{m+2k}
\end{align*}
and we see that
\[
a_{0,k} =  b_{0,k-1} \quad\text{and}\quad a_{jk} = b_{j,k-1}(2j+1) + b_{j-1,k-1},\ 1 \leq j \leq k-1 \quad\text{and}\quad a_{kk} = b_{k-1,k-1} = 1.
\]
It is immediate that the coefficients grow at worst like $C A^k k!$ for some constants $A,C>0$. 
\end{proof}

\begin{proof}[Proof of Theorem \ref{thm:surface area measure is global}]
The calculation of $\widehat{d\sigma}(\xi)$ is well known and can be found in \cite[Chapter VIII.3]{Ste93}.
For large $r$, we know that
\[
J_m(r) = \Big(\frac 2\pi\Big)^{1/2} r^{-1/2} \cos(r-\pi m/2 - \pi/4) + O(r^{-3/2}).
\]
We start by computing the range of $q$ for which $f_{\frac{d-2}2}(|\xi|) \in L^q(\R^d)$. Since 
$|\cos r| \geq \frac 12$ if $|r-k\pi| \leq \frac \pi 3$,  for $N$ sufficiently large and $q \geq 1$, it follows that
\begin{align*}
\int_{\R^d} |f_{\frac{d-2}2}(|\xi|)|^q\, dV(\xi)
&\geq \frac1{100^q}\sum_{k=N}^\infty \int_{|r-k\pi| \leq \frac\pi3} \frac{1}{r^{q(\frac 12+\frac{d-2}2)-(d-1)}} \, dr \\
&\geq \frac{1}{100^q} \sum_{k=N}^\infty \frac{1}{[(k-\frac 13)\pi]^{(d-1)(\frac q2-1)}}
\end{align*}
This sum converges exactly when $(d-1)(\frac q2-1)>1$, that is, the sum converges if and only if
$q > \frac{2d}{d-1}$.
Similarly, for $R$ large,
\[
\int_{\R^d} |f_{\frac{d-2}2}(|\xi|)|^q\, dV(\xi)
\leq C \int_{|\xi|\geq R} \frac{1}{|\xi|^{q(\frac 12 + \frac{d-2}2})}\, dV(\xi) 
= C \int_R^\infty \frac{1}{r^{(d-1)(\frac q2 - 1)}}\, dr.
\]
The range of $q$ for which the latter integral converges is again $q > \frac{2d}{d-1}$.

The terms in \eqref{eqn:form of f'} with the worst decay in $r$ corresponds to $r^{2k}f_{m+2k}$ 
and $r^{2k+1}f_{m+2k+1}$, both of which satisfy $r^\ell f_{m+\ell} \sim O(r^{-m-1/2})$. This means $f_m^{(j)}(|\xi|)\in L^q(\R^d)$ for the same range of $q$ as $f_m$. Moreover,
since $f_m$ is analytic and even, it follows that for some constants $C$ and $A$, if $q > 2 + \frac 2{d-1}$, then
\[
\|f^{(k)}_{\frac{d-2}2}\|_{L^q(\R^d)} \leq C_q A^k k^k. 
\]
The mapping $\xi \mapsto f(|\xi|)$ is even in $\xi$, real analytic in $|\xi|$ (and hence well behaved on $B(0,1)$), and $\xi \mapsto |\xi|$ is easily shown to be a function in
$\opG^{\infty,1}(B(0,1)^c)$. Thus, Part 1 of Theorem \ref{thm:surface area measure is global} 
now follows from the arguments in \cite[Section 5]{AdHoRa17}.

Let $\alpha$ be a multi-index and $d\mu_\alpha = x^\alpha\, d\sigma$. $d\mu_\alpha$ is not a positive measure unless $\alpha$ has only multi-indices, but that
will not prevent us from using the Tomas-Stein restriction theorem arguments \cite{Tom75,Moc00}. We want to show that there exists constants
$C_p,A_p>0$ so that if $T^\alpha$ is the operator defined by convolution against $D^\alpha \widehat{d\sigma}$, then
\[
\|T^\alpha f\|_{L^{p'}(\R^d)} \leq C_p A_p^{|\alpha|}|\alpha|^{|\alpha|}
\] 
Let $\{\vp_k:k=0,1,2,\dots\}\subset C^\infty_c(\R^d)$
be a partition of unity so that $\supp \vp_0\subset B(0,2)$, and for $k\geq 1$, $\supp\vp_k\subset \{x : 2^{k-1} \leq |x| \leq 2^{k+1}\}$ and satisfies
$\vp_k(x) = \vp(x/2^k)$ for some $\vp \in C^\infty_c(\R^d)$. 
Set $T_k f = (\vp_k D^\alpha\widehat{d\sigma})*f$ so that $\sum_{k=0}^\infty T_k = T$.

It follows immediately proof of Theorem \ref{thm:surface area measure is global} that
there exist constants $C,A>0$ so that
\[
\|T_k\|_{L^1\to L^\infty} \leq \|\vp_k D^\alpha \widehat{d\sigma}\|_{L^\infty(\R^d)} \leq C A^{|\alpha|} 2^{-k\frac{d-1}2}.
\]
Next, we bound $\|T_k\|_{L^2\to L^2}$ by Plancherel's Theorem.
Using the fact that $\hat\vp\in\opS(\R^d)$ and decays rapidly, for some constant $C$ (that can increase line by line but independent of $k$)
\begin{align*}
\|T_k\|_{L^2\to L^2}
&= \sup_{\atopp{f,g\in \opS(\R^d)}{\|f\|_{L^2} = \|g\|_{L^2}=1}} \big|\big( (\hat \vp_k* d\mu_\alpha)\hat f, \hat g\big) \big| \\
&\leq \sup_{x\in\R^n} |\hat \vp_k* d\mu_\alpha(x)| \\
&\leq C 2^{kd} \sup_{x\in\R^d}  \int_{\R^d} \frac{1}{(1+2^k|x-y|)^d}\, |d\mu_\alpha|(y) \\
&\leq C 2^{kd} \sup_{x\in\R^d}  \int_{S^{d-1}} \frac{1}{(1+2^k|x-y|)^d}\, d\sigma(y) 
\end{align*}
It is clear that the supremum occurs whenever $x\in S^{d-1}$, and at $x=1$, in particular. We compute that
\begin{align}
\int_{S^{d-1}} \frac{1}{(1+2^k|1-y|)^N}\, d\sigma(y)  \nn
&\leq \int_{|1-y|< \frac 1{2^k}} d\sigma(y) + \sum_{j=0}^{k-1}\Big\{ \int_{\frac{2^j}{2^k} \leq|1-y|<\frac{2^{j+1}}{2^k}} \Big(2^k\frac{2^j}{2^k}\Big)^{-d}\,d\sigma(y)\Big\}
+ \frac{C}{2^{kd}} \\
&\leq \frac{C}{2^{k(d-1)}}.\label{eqn:eliminating alpha}
\end{align}
Thus, $\|T_k\|_{L^2\to L^2} \leq C 2^k$. We use the Riesz-Theorem Interpolation Theorem and bound 
$\|T_k\|_{L^p\to L^{p'}} \leq C A^{|\alpha|\frac{2-p}p} 2^{-\frac kp(2(d+1)-p(d+3))}$ Since $T^\alpha = \sum_{k=0}^\infty T_k$, this sum will converge
when $1 \leq p < \frac{2(d+1)}{d+3}$ and the theorem is proved.
\end{proof}

\bibliographystyle{alpha}
\bibliography{mybib}
\end{document}